\documentclass[11pt]{article}
\usepackage{amsmath, amsfonts, amsthm, amssymb,  multicol}
\usepackage{graphicx}
\usepackage{float}
\usepackage{verbatim}
\usepackage{color}

\numberwithin{equation}{section}

\newtheorem{thm}{Theorem}[section]
\newtheorem{lma}[thm]{Lemma}
\newtheorem{cor}[thm]{Corollary}

\renewcommand{\geq}{\geqslant}
\renewcommand{\leq}{\leqslant}

\renewcommand{\epsilon}{\varepsilon}

\newcommand{\lld}{\underline{\dim}_{\mathrm{loc}}}

\newcommand{\ad}{\dim_{\mathrm{A}}}
\newcommand{\bd}{\dim_{\mathrm{B}}}
\newcommand{\hd}{\dim_{\mathrm{H}}}

\newcommand{\N}{\mathbb{N}}

\title{The Assouad dimension of self-affine carpets \\ with no grid structure}

\author{Jonathan M. Fraser \& Thomas Jordan
 }

\begin{document}
\maketitle

\begin{abstract}
Previous study of the Assouad dimension of planar self-affine sets has relied heavily on the underlying IFS having a `grid structure', thus allowing for the use of approximate squares.  We study the Assouad dimension of a class of self-affine carpets which do not  have an associated grid structure.  We find that the Assouad dimension is related to the box and Assouad dimensions of the (self-similar) projection of the self-affine set onto the first coordinate and to the local dimensions of the projection of a natural Bernoulli measure onto the first coordinate.  In a special case we relate the Assouad dimension of the Przytycki-Urba\'nski sets to the lower local dimensions of Bernoulli convolutions.   \\


\end{abstract}

\section{Introduction}

The Assouad dimension is an important concept in embedding problems and metric geometry and is becoming increasingly popular in the study of fractals.  It is larger than the other commonly used notions of dimension, such as Hausdorff and box-counting dimension, and gives a coarse indication of how thick the thickest part of a metric space is.  In 2011, Mackay \cite{mackay} computed the Assouad dimension of a simple but fascinating family of self-affine sets known as Bedford-McMullen carpets.  These sets were first considered by Bedford \cite{bedford} and McMullen \cite{mcmullen} in the mid-1980s and have been intensively studied since then. One of the key reasons for their popularity is that, despite exhibiting many of the complications and key features of  self-affine sets (such as non-conformal distortion on small scales and the possibility of  distinct Hausdorff and box-counting dimensions), they have a very simple grid structure which facilitates analysis and makes them much more tractable than general self-affine sets.   Mackay's work was subsequently extended in \cite{fraser, fraserhowroyd} to more general classes of self-affine set, but the calculations and dimension formulae still relied heavily on a grid structure.  In this paper we drop the dependence on a grid structure for the first time and in doing so relate the Assouad dimension to, both geometric and symbolic, local dimensions of certain Bernoulli measures.  A particularly demonstrative example is provided by the self-affine sets considered by Przytycki-Urba\'nski  \cite{PU} and here the Assouad dimension can be computed in terms of the infimal local dimension of the associated Bernoulli convolution.  We discuss this particular example in detail and reveal that different phenomena are present depending on the number theoretic properties of the parameter.

\subsection{The Assouad dimension}

Let $X = [0,1]^2$.  For any non-empty subset $E \subseteq X$ and $r>0$, let $N_r (E)$ be the smallest number of open sets with diameter less than or equal to $r$ required to cover $E$.  The \emph{Assouad dimension} of a non-empty set $F \subseteq X$ is defined by
\begin{eqnarray*}
\ad F &=&  \inf \Bigg\{ s \  : \ \text{      $ (\exists \, C>0)$ $(\forall \, 0<r<R)$ $(\forall \, x \in F)$} \\ \\
&\, & \hspace{40mm}  N_r\big( B(x,R) \cap F \big) \ \leq \ C \bigg(\frac{R}{r}\bigg)^s  \Bigg\}.
\end{eqnarray*}
For more information on the Assouad dimension including some discussion of its basic properties, we refer the reader to \cite{luk, robinson, fraser}.  For the purposes of this work it is convenient to point out that one may replace $B(x,R)$ with a small square of sidelength comparable to $R$ and the function $N_r( \cdot)$ with any related covering or packing function using small balls, general open sets, or small squares, all with diameter comparable to $r$.

\subsection{Self-affine carpets}

Roughly speaking, self-affine carpets are planar self-affine sets generated by a finite family of contracting diagonal matrices. More precisely, let $\mathcal{I}$ be a finite index set and $\{S_i\}_{i \in \mathcal{I}}$ be a collection of affine maps which map the unit square $X$ into itself and have the form
\[
S_i(x,y) = (b_ix,a_iy)+ \underline{t}_i
\]
for constants $a_i, b_i \in (0,1)$ and a translation vector $\underline{t}_i \in [0,1-b_i] \times [0, 1-a_i]$.  Such a collection of maps is called an \emph{iterated function system} (IFS) and it is well-known that there is a unique non-empty compact set $F$ satisfying
\[
F = \bigcup_{i \in \mathcal{I}} S_i(F),
\]
see \cite[Chapter 9]{falconer}.  The set $F$ is called the \emph{attractor} of the IFS and is the associated \emph{self-affine carpet}.  Computing the dimensions of such self-affine carpets is challenging, due to the fact the maps contract by different amounts in different directions, but is considerably more tractable than the case of general self-affine sets, where the diagonal matrices are replaced with general non-singular contracting matrices.

We assume that the interiors of the rectangles $S_i(X)$ are pairwise disjoint; a condition sometimes referred to as the rectangular open set condition.   Often one requires additional conditions on the defining IFS in order to obtain explicit  dimension formulae.  One such condition which is particularly relevant to our work is the requirement that the IFS has a `grid structure'.  Such self-affine sets were considered by Bedford-McMullen, as well as Bara\'nski \cite{baranski} and Lalley-Gatzouras \cite{lalley-gatz}.  What this means is that the rectangles $S_i(X)$ are aligned in such a way that when they are projected onto the coordinate axes they either fall precisely on top of each other or have disjoint interiors.  This allows one to introduce `approximate squares', which greatly facilitate dimension estimates and forces the coordinate projections to be self-similar sets satisfying the open set condition.  Assuming this grid structure, the Assouad dimensions of these carpets were computed by Mackay \cite{mackay} and Fraser \cite{fraser}.  What has emerged is that the Assouad dimension is given by
\[
\ad F = \dim \pi F + \max_C \dim \text{C}
\]
where $\pi$ is the projection onto the relevant coordinate axis and the maximum is taken over an explicit finite collection of sets $C$ induced by the columns given by the grid structure in the direction of the fibers of $\pi$.  It seemed plausible that in the absence of a grid structure a similar formula may hold, but it was unclear what  should play the role of the column sets.  It was also unclear which dimension `$\dim$' should refer to, since in the grid case all the relevant dimensions are equal.  Our work sheds light on both of these issues.

\section{Results}
We consider the self-affine carpets described in the previous section, under the additional assumption that for all $i \in \mathcal I$ we have $a_i=\alpha < \beta = b_i$ for fixed constants $\alpha, \beta \in (0,1)$. For more discussion of this assumption, see Section \ref{further}.  We emphasise that we do not require a grid structure and the projection $\pi F \subseteq [0,1]$ is a self-similar set which may have complicated overlaps.

Let $\mathcal{I}^* = \bigcup_{k\geq1} \mathcal{I}^k$ denote the set of all finite sequences with entries in $\mathcal{I}$ and for $\textbf{i}= \big(i_1, i_2, \dots, i_k \big) \in \mathcal{I}^*$ write
\[
S_{\textbf{i}} = S_{i_1} \circ S_{i_2} \circ \dots \circ S_{i_k}
\]
noting that this map contracts by $\beta^k$ horizontally and $\alpha^k$ vertically.  Also write $\lvert \textbf{i} \rvert = k$ to denote the length of the word $\textbf{i}$.

 Let $m = \lvert \mathcal I \rvert$ and $\mu$ be the unique Borel probability measure on $F$ satisfying
\[
\mu \left( S_\textbf{i} (X) \right) = m^{-k}
\]
for all $\textbf{i} \in \mathcal{I}^k$.  It is easy to see that this measure is well-defined and simply corresponds to the push forward of the uniform Bernoulli measure on the associated shift space under the natural coding map.

  Let $\pi$ denote orthogonal projection from $X$ onto the first coordinate and write $\pi \mu = \mu \circ \pi^{-1}$ for the push-forward of $\mu$ under $\pi$.  It is clear that $\pi \mu$ is a self-similar measure supported on the self-similar set $\pi F$.   The measure $\pi\mu$ will be key to our analysis and in particular the quantity
\[
s\ :=\ \sup\{t:\exists C>0\text{ with }\pi\mu(I)\leq C|I|^t\text{ for all subintervals } I\subseteq [0,1]\}.
\]
It should be noted that $s$ is non-negative and bounded above by the lower local dimension at any point $x$, i.e.
\[
0 \, \leq \,  s \, \leq  \, \underline{\dim}_\text{loc} (\pi \mu , x) \ :=  \ \liminf_{r\to 0}\frac{\log\pi\mu(B(x,r))}{\log r}
\]
for all $x\in [0,1]$.  We will argue below that this upper bound for $s$ is actually an equality in our setting.  The minimal lower local dimension is bounded above by the Hausdorff and box dimension of the support of the measure and thus
\[
 s  \leq \bd \pi F \leq \min \left\{\frac{\log m}{-\log \beta}, \, 1\right\}.
\]
We also observe that $s$ can be expressed in terms of the $L^q$-spectrum of $\pi\mu$.   The $L^q$-spectrum is a standard tool in multifractal analysis and information theory and is defined as follows for a given Borel probability measure $\nu$.  For $q >0$ and $r>0$ let
\begin{eqnarray*}
M_r^q(\nu) &=& \sup \left\{ \sum_{i} \nu(U_i)^q : \{U_i\}_i \text{ is a centered packing of $\text{supp} \, \nu$ by balls of radius $r$ }\right\}
\end{eqnarray*}
and then the $L^q$-spectrum of $\nu$ is defined by
\[
\tau_{ \nu}(q) = \lim_{r \to 0} \frac{\log M_r^q(\nu)}{\log r}
\]
provided this limit exists. It was proved by Peres and Solomyak \cite{exists}  that the above limit exists for all $q>0$ whenever $\nu$ is a self-similar measure, and so in particular $\tau_{ \pi \mu}(q)$ exists for all $q>0$.  It is straightforward to see that the $L^q$-spectrum  is non-decreasing and concave in $q$.  We observe that $s$ is equal to the slope of the asymptote as $q \to \infty $.  This result holds much more generally than just for self-similar measures.
\begin{lma}\label{lq}
For $s$ defined above and $\tau_{ \pi \mu}$ the $L^q$-spectrum of the projected measure $\pi \mu$, we have
\[
s \ = \ \inf  \left\{ t \geq 0 \ : \    \tau_{ \pi \mu}(q) <  tq  \ \text{\emph{for all $q>0$}}  \right\}.
\]
\end{lma}

\begin{proof}
 Let $t>0$ and $q > 0$ be such that $\tau_{ \pi \mu}(q) = t q$ and let $\varepsilon>0$.  Then there exists a uniform constant $C>0$ such that
\[
M_r^q(\pi \mu) \  \leq  \  C r^{(t-\varepsilon)q}.
\]
Then, since $\{B(x,r)\}$ is a (rather trivial) $r$-packing for any particular choice of $x$ we have the following estimate which is uniform in $x$
\[
\mu\left( B(x,r) \right)^q \  \leq \ M_r^q(\pi \mu)  \  \leq  \  C r^{(t-\varepsilon)q}
\]
which proves that $s\geq t-\varepsilon $ and letting $\varepsilon \to 0$ completes the proof that
\[
s \ \geq\ \inf  \left\{ t \geq 0 \ : \    \tau_{ \pi \mu}(q) <  tq  \ \text{for all $q>0$} \} \right\}.
\]

 To prove the inequality in the other direction we suppose that there exist $C,t>0$ such that $\pi\mu(I)\leq C|I|^t$ for all $I\subset [0,1]$. We then suppose that $\{U_i\}$ is a centered packing of $\text{supp} \, \pi\mu$ by balls of radius $r$. It follows that there are at most $r^{-1}$ elements in $\{U_i\}$ and 
$$\sum_{i}\pi\mu(U_i)^q\leq Cr^{tq-1}.$$
Thus taking the supremum over all such centrered packings yields,
$$M_r^q(\pi\mu)\leq Cr^{tq-1}$$
and thus
$$\tau_{ \pi\mu}(q)\geq tq-1.$$
Thus for all $q >0$ we have that
$$\frac{\tau_{ \pi\mu}(q)}{q}\geq \frac{tq-1}{q}.$$
Thus if we choose $t_1<t$ then there will exist $q>0$ such that $\frac{\tau_{ \pi\mu}(q)}{q}>t_1$ and so 
\[
s \ \leq\ \inf  \left\{ t \geq 0 \ : \    \tau_{ \pi \mu}(q) <  tq  \ \text{for all $q>0$} \right\}
\]
as required.
\end{proof}

If the multifractal formalism  holds for all $q> 0$ it follows that $s$ is actually equal to the infimal lower local dimension.  Feng and Lau \cite{Fengmf} showed that this is the case for self-similar measures satisfying the weak separation condition for example, but we can say more.   Theorem 1.1 from \cite{feng0} states that, in the case of self-similar measures, the multifractal formalism holds for $q \geq 1$ whenever the $L^q$-spectrum is differentiable at $q$.  Since the $L^q$-spectrum is non-decreasing and concave, we can find a sequence of $q \to \infty$ along which the $L^q$-spectrum is differentiable and, moreover, the derivatives converge to $s$, by the above lemma.  Thus we may conclude that
\[
s \ = \  \inf  \left\{ t \geq 0 \ : \    \tau_{ \pi \mu}(q) <  tq  \ \text{for all $q>0$} \} \right\} \ = \ \inf_{x \in \pi F} \underline{\dim}_\text{loc} (\pi \mu , x) .
\]
Finally, it follows from Proposition 2.2 of \cite{Fengmf} that this value is strictly positive, provided the set $\pi F$ is not a singleton.

Define an equivalence relation on $\mathcal I^*$ by $\textbf{i} \sim \textbf{j}$ if and only if
\[
\pi S_\textbf{i}(X) = \pi S_\textbf{j}(X)
\]
and denote the equivalence class of $\textbf{i}$ by $ [\textbf{i}]$.  Note that $\textbf{i} \sim \textbf{j} \Rightarrow \lvert \textbf{i} \rvert = \lvert \textbf{j} \rvert $.  Let
\[
H \ :=  \  \sup_{\textbf{i} \in \mathcal I^*} \frac{\log \lvert [\textbf{i}] \rvert }{\lvert \textbf{i} \rvert },
\]
where $\lvert [\textbf{i}] \rvert$ denotes the cardinality of $[\textbf{i}]$, and observe that $0 \leq H \leq \log m$.  If the underlying IFS of similarity maps defining $\pi F$ generates a free semigroup, then $[\textbf{i}] = \{ \textbf{i}\}$ for all $\textbf{i}$, which renders $H = 0$, and if the semigroup is not free, then $H>0$.  Here
\[
\frac{\log m- H}{-\log \beta} \geq s
\]
is the `minimal symbolic local dimension' of $ \pi \mu$.   Note that
\[
0 \leq \frac{H}{-\log \alpha} \leq \frac{\log m \beta^s}{-\log \alpha}  \leq \frac{\log m}{-\log \alpha}.
\]

We are now ready to state our main result, which bounds the Assouad dimension of $F$ in terms of the parameters defined above.

\begin{thm} \label{main}
Let $F$ be a self-affine carpet in the class defined above.  Then
\begin{eqnarray*}
\max \left\{ \bd \pi F \, + \, \frac{\log m \beta^s}{-\log \alpha},  \ \  \ad \pi F  \, + \, \frac{H}{-\log \alpha} \right\}  \ \leq \   \ad F \ \leq \ \ad \pi F \, + \, \frac{\log m \beta^s}{-\log \alpha}.
\end{eqnarray*}
\end{thm}

There are several concrete settings where we can obtain a sharp result.  First recall that the \emph{weak separation property}, see \cite{lau,zerner}, is a weaker condition than the open set condition and applies to self-similar sets with `controllable overlaps'.  In particular, the main result of \cite{fraseretal} gives that if  $\pi F$ satisfies the weak separation property, then  $\bd \pi F = \ad \pi F$, and if the weak separation property fails then $\ad \pi F = 1$.

\begin{cor}
Let $F$ be a self-affine carpet in the class defined above.
\begin{enumerate}
\item  If $\bd \pi F = 1$, for example if $\pi F = [0,1]$, then
\[
 \ad F \ = \ 1 \, + \, \frac{\log m \beta^s}{-\log \alpha}.
\]
\item  If $\pi F$ satisfies the weak separation property, then
\[
 \ad F \ = \ \bd \pi F \, + \, \frac{\log m \beta^s}{-\log \alpha}.
\]
\item  If $s$ is given by the minimal symbolic local dimension of $\pi \mu$, i.e. if
\[
s = \frac{\log m- H}{-\log \beta} ,
\]
then
\[
 \ad F \ = \ \ad \pi F \, + \, \frac{H}{-\log \alpha}.
\]
\end{enumerate}
\end{cor}

At this point we highlight that we are not aware of any examples which do not fall under cases \emph{1.,  2.} or \emph{3.} in the above corollary and it is plausible that no such examples exist. In fact Theorem 6.6 in the paper \cite{shmerkin}\footnote{This result appeared between the first and second versions of this paper.} gives a large class of examples which either fall into case \emph{1.} with $s=1$ or case \emph{3.} with $H=0$.   

Both cases \emph{1.} and \emph{2.} could be subsumed into a single condition: $\bd \pi F = \ad \pi F$, but we state them separately because one would go about checking them in different ways.  Also, sets satisfying \emph{1.} may have complicated overlaps in the projection and we want to emphasise that we can still explicitly handle this situation in some cases.

\subsection{Applications to Bernoulli convolutions and Przytycki-Urba\'nski sets}

In order to give a concrete example, as well as demonstrate some interesting properties, we consider a special family of self-affine carpets in our class.  This family is that considered by Przytycki and Urba\'nski in \cite{PU} and is heavily related to Bernoulli convolutions, see \cite{solomyak, PSS}.

Fix  $0<\alpha \leq 1/2 < \beta < 1$, let $m=2$, and consider the IFS consisting of the maps
\[
S_1(x,y) = (\beta x, \alpha y) \qquad \text{and} \qquad S_2(x,y) = (\beta x+1-\beta, \alpha y+1-\alpha).
\]
Here $\pi \mu$ is the well-studied $\beta$-\emph{Bernoulli convolution}, $\nu_\beta$, i.e. the (rescaled) distribution of the random series $\sum \pm \beta^k$ where the signs are chosen independently and without bias,  and $\pi F$ is simply the unit interval.  As such, we obtain a sharp result for the Assouad dimension of $F$:
\[
 \ad F \ = \ 1 \, + \, \frac{\log 2 \beta^s}{-\log \alpha}.
\]
It is known that for any such $\alpha, \beta$ the box (and packing) dimensions of $F$ are given by
\[
 \bd F \ = \ 1 \, + \, \frac{\log 2 \beta}{-\log \alpha}
\]
and this, therefore, coincides with the Assouad dimension if and only if $s=1$, and otherwise is strictly smaller. Also for any such $\alpha, \beta$ the Hausdorff dimension of $F$ can be bounded below by 
\[
 \hd F \ \geq \ \hd \mu \ = \ \hd \nu_\beta \, + \, \frac{\log 2 \beta}{-\log \alpha}.
\]
Thus if $\dim \nu_\beta=1$ then $\hd F=\bd F$. On the other hand if $\dim \nu_{\beta}<1$ then it is possible that $\hd F<\bd F$ (in all cases where it is known that $\dim \nu_{\beta}<1$ this is the case). Recall that the (lower) Hausdorff dimension of a measure is equal to the essential infimum of the lower local dimensions and so if $\dim \nu_{\beta}<1$ then $\ad F>\bd F$.  All of the measures we discuss here are exact dimensional and so the lower Haudorff dimension also coincides with the upper packing dimension which is the essential supremum of the upper local dimensions.  In general it is not known whether there exists $\beta$ where $\hd \nu_{\beta}=1$ but where $s<1$.

As is common with the study of Bernoulli convolutions, it is natural to consider different special cases related to the number theoretical properties of $\beta$:

\begin{enumerate}
\item  If $1/\beta$ is \emph{Garsia}, i.e., a real algebraic integer with norm 2 whose conjugates lie strictly outside the unit disc, then the Bernoulli convolution is absolutely continuous with bounded density (see \cite{garsia}) and therefore $\hd \nu_\beta = s=1$ and 
\[
1< \hd F = \bd F = \ad F = 1 \, + \, \frac{\log 2 \beta}{-\log \alpha} < 2.
\]
\item If $1/\beta$ is \emph{Pisot}, i.e., a real algebraic integer greater than 1 whose conjugates lie strictly inside the unit disc, then it follows from Theorem 8 in \cite{PU} that $1<\hd F<\bd F$ and $\hd\nu_{\beta}<1$. Thus $s<1$ and we can conclude that
\[
1<\hd F < \bd F < \ad F<2 .
\]
In fact if $\beta_k$ is the positive solution to $\sum_{n=1}^kx^n=1$ then $\beta_k^{-1}$ is Pisot and it is shown in Theorem A of \cite{hu} that
\[
\inf_{x \in [0,1]}  \lld ( \nu_{\beta_k}, x) \ = \ \frac{\log \phi }{k \log \beta_k} - \frac{\log 2}{\log \beta_k}
\]
where $\phi$ is the Golden Ratio. 

\item If, for $n\geq 4$, $\gamma$ is the largest real root of $x^n-x^{n-1}-\cdots-x+1$,  then $\gamma$ is \emph{Salem},  i.e. a real algebraic integer greater than 1 whose conjugates lie inside the unit disc. If we let $\beta=\gamma^{-1}$ then it was shown by Feng  \cite[Theorem 1.2]{feng}  that $s<1$ and so in this case
\[
1<\bd F < \ad F  < 2.
\]
In this case it is unknown whether $\hd  \nu_{\beta}=1$ and whether $\bd F=\hd F$ (if $\hd  \nu_{\beta}=1$ then it follows that $\bd F=\hd F$, see above). 
\item
For almost every $\beta\in [2^{-1/2},1)$ it is known that $\nu_{\beta}$ is absolutely continuous with bounded density, see \cite[Corollary 1]{solomyak}. It has also recently been shown in Theorem 1.5 in \cite{shmerkin} that there exists a set $\mathcal{E}\subset (1/2,1)$ with packing dimension $0$ where for all $\beta\in (1/2,1)\backslash\mathcal{E}$ we can take $s=1$. Therefore for all $\beta\in (1/2,1)$, except for a set of packing dimension $0$, we have that
 \[
1< \hd F = \bd F = \ad F = 1 \, + \, \frac{\log 2 \beta}{-\log \alpha} < 2.
\]
\end{enumerate}

For a given $\beta\in (1/2,1)$ let $s(\beta)$ be the value of $s$ for the measure $\nu_\beta$.  We can give more precise estimates on $s(\beta)$, and thus on $\ad F$, using the convolution structure of $\nu_\beta$ and results from \cite{JSS}.   Indeed, the name `Bernoulli convolution' comes from the fact that $\nu_\beta$ is the (appropriately rescaled) infinite convolution of the measures $\tfrac{1}{2}(\delta_{-\beta^k}+\delta_{\beta^k})$. This approach was suggested to us by P\'eter Varj\'u.

\begin{lma}
For any $\beta \in (1/2,1)$ and $n \in \mathbb{N}$ we have $s(\beta)\geq s(\beta^n)$.  In particular, this shows that
\begin{enumerate}
\item $s(\beta) > 1/2$ for all $\beta \in (1/2,1)$
\item $s(\beta) \to 1$ as $\beta \to 2^{-1/n}$ for any $n \in \mathbb{N}$.
\end{enumerate}
\end{lma}

\begin{proof}
This result essentially follows from the fact that taking the convolution of two measures does not decrease the value of $s$.    Begin by writing $\nu_\beta$ as the convolution of $\nu_{\beta^n}$ and another measure $\nu'$ which comes from taking the (appropriately rescaled) infinite convolution of the measures $\tfrac{1}{2}(\delta_{-\beta^k}+\delta_{\beta^k})$ for values of $k$ not divisible by $n$.  Let $t<s(\beta^n)$, which guarantees the existence of a constant $C>0$ such that for any interval $I \subseteq [0,1]$ we have $\nu_{\beta^n}(I) \leq C\lvert I \rvert^{t}$.  As such, for any interval $I \subseteq [0,1]$ we have
\[
\nu_\beta(I) \ = \ \left(\nu_{\beta^n} \ast \nu'\right)(I) \ = \ \int \nu_{\beta^n}(I-x) d \nu'(x) \ \leq \ \int  C\lvert I-x \rvert^{t} d \nu'(x) \ = \  C\lvert I \rvert^{t}
\]
and therefore $s(\beta) > t$ and letting $t$ tend to $s(\beta^n)$ proves the result.  Now, for a given $\beta$, choose $n=n(\beta) \in \mathbb{N}$ uniquely to satisfy $\beta^n \leq 1/2<\beta^{n-1}$ and note that $\nu_{\beta^n}$ is a self-similar measure satisfying the open set condition and so the value of $s(\beta^n)$ can easily be computed as
\[
s(\beta^n) = \frac{\log 2}{-n \log \beta} > 1/2
\]
proving \emph{1.}  Part \emph{2.} follows immediately from the estimate $s(\beta)\geq s(\beta^n)$ and the fact that $\lim_{\beta\to 1/2}s(\beta)=1$, which follows from \cite[Theorem 1.11]{JSS}.  
\end{proof}

The estimates from the previous lemma provide further estimates for $\ad F$.

\begin{cor}
For all $\beta\in (1/2,1)$ we have
\[
1 \ < \  \ad F \ = \ 1 \, + \, \frac{\log 2 \beta^{s(\beta)}}{-\log \alpha} \ < \ 1+\frac{\log 2}{-\log\alpha} -\frac{\log\beta}{2\log\alpha}  \ < \ 2
\]
and  $ \ad F \to  \bd F$ as $\beta\to 2^{-1/n}$ for any $n \in \mathbb{N}$.
\end{cor}

Closer inspection of the proof of \cite[Theorem 1.11]{JSS} would yield explicit upper bounds for $\ad F$ in terms of $\bd F$ and $\beta$ for $\beta$ in a neighbourhood of $\beta =  2^{-1/n}$, but we do not include the details.

Finally, a   potential application of these results is to gain information about the local dimensions of Bernoulli convolutions based on \emph{a priori} information about the Assouad dimension of $F$.  Consider the case where $\alpha = 1/2$.  Then
\[
s \ =  \  (2-\ad F) \frac{\log 2}{-\log \beta}
\]
and so, for example, if one knew that $\ad F \leq \theta$, then one could conclude that for all $x \in [0,1]$
\[
 \lld ( \nu_\beta, x) \ \geq \ s \ \geq \ (2-\theta) \frac{\log 2}{-\log \beta}.
\]

\section{Proofs}

The proof of Theorem \ref{main} will be broken up into three parts.  The first two parts together prove the lower bound: first the bound where $\bd \pi F$ appears in the formula, and second the bound where $\ad \pi F$ appears.  The final part will prove the upper bound.  Each of the following subsections is self-contained and notation introduced in one section does not carry over into the others.

For real-valued functions $f,g$, we write $f(x)\lesssim g(x)$ to mean that there exists a universal constant $M>0$ independent of $x$ such that $f(x)\leq Mg(x)$. Similarly, $f(x)\gtrsim g(x)$ means that $f(x)\geq Mg(x)$ with a universal constant independent of $x$.  If both $f(x)\lesssim g(x)$ and $f(x)\gtrsim g(x)$, then we write $f(x)\asymp g(x)$.  The constant $M$ can depend on parameters that are fixed by the defining IFS, such as $m$, $\alpha$, $\beta$, $s$, $H$, etc.

\subsection{Lower bound}

\subsubsection{Picking up the box dimension of $\pi F$}

Throughout this section let $d = \bd \pi F$.   In order to prove the first lower bound, it suffices to show that for all $\varepsilon>0$ there exist arbitrarily small $0<r<R$ with $(R/r) \to \infty$ such that for some square $Q$ intersecting $F$ with side length comparable to $R$ we have
\[
N_r(Q \cap F) \gtrsim (R/r)^{d - \log (m \beta^s)/\log \alpha - \varepsilon}.
\]
It will be convenient to consider pairs $R = \alpha^k$ and $r = \alpha^{n(k)+k}$ for a particular sequence of $k \to \infty$,  where $n(k)  = n \in \mathbb{N}$ is chosen (uniquely) to satisfy
\[
\beta^{n+1} < (\alpha/\beta)^k \leq \beta^n.
\]
Observe that $r \asymp R^{\log\alpha/\log \beta}$ and indeed $(R/r) \to \infty$ as $k \to \infty$.  With this in mind note that, by the definition of $s$, for any $\epsilon>0$ we can find sequences $(k_j)_{j \in\N}$ (with $k_j \to \infty$) and $(x_j)_{j \in\N}$ (with $x_j \in \pi F$) such that
\[
\pi\mu(B(x_j,(\alpha/\beta)^{k_j}))\geq (\alpha/\beta)^{k_j(s+\epsilon)}.
\]
Consider the vertical tube $T = \pi^{-1}( B(x_j,(\alpha/\beta)^{k_j}/2)) \cap X$ and let
\[
M(n(k_j)) = \# \{ \textbf{i} \in \mathcal I^{n(k_j)} : S_{\textbf{i}}(F) \cap T \neq \emptyset \}
\]
be the number of level $n(k_j)$ cylinders intersecting $T$.  Since each level $n(k_j)$ cylinder has $\mu$-mass $m^{-n(k_j)}$ we have
\[
\pi \mu \left( B(x_j,(\alpha/\beta)^{k_j}/2 )\right) \ \leq \ M(n(k_j)) m^{-n(k_j)}
\]
which yields
\[
M(n(k_j)) \ \geq \ m^{n(k_j)} (\alpha/\beta)^{k_j(s+\varepsilon)}.
\]
Choose an arbitrary $\textbf{i} \in \mathcal I^{k_j}$ and consider $Q = S_{\textbf{i}}(T)$, which is an $\alpha^{k_j}$ by $\alpha^{k_j}$ square intersecting $F$.  Consider $6Q$ defined to be the square formed by blowing $Q$ up by a factor of $6$ at its center.  It follows that every level $(n(k_j)+k_j)$ cylinder which intersects $Q$ is completely contained within $6Q$.   Let $\textbf{i} \in \mathcal I^{n(k_j)+k_j}$ and consider the problem of trying to cover $S_\textbf{i}(F)$ by squares of sidelength $\alpha^{n(k_j)+k_j}$.  By rescaling the problem by $\beta^{-(n(k_j)+k_j)}$ one observes that
\[
N_{\alpha^{n(k_j)+k_j}} \big( S_\textbf{i}(F) \big)  \ = \    N_{(\alpha/\beta)^{n(k_j)+k_j}} \big( \pi F \big) \  \gtrsim \   (\alpha/\beta)^{-(n(k_j)+k_j)(d-\varepsilon)} .
\]
with the implied constant independent of $k_j$.  Since one $\alpha^{n(k_j)+k_j}$ square can intersect no more than 9 of the level $(n(k_j)+k_j)$ cylinders (by virtue of our separation condition), and observing that by definition
\[
M(n(k_j)) \asymp \# \{ \textbf{i} \in \mathcal I^{n(k_j)+k_j} : S_{\textbf{i}}(F) \cap Q \neq \emptyset \}
\]
 we have
\begin{eqnarray*}
N_{\alpha^{n(k_j)+k_j}} (6Q \cap F) &\gtrsim&  M(n(k_j))  (\alpha/\beta)^{-(n(k_j)+k_j)(d-\varepsilon)} \\ \\
 &\geq&  m^{n(k_j)} (\alpha/\beta)^{k_j(s+\varepsilon)} (\alpha/\beta)^{-(n(k_j)+k_j)(d-\varepsilon)} \\ \\
&\gtrsim& m^{n(k_j)} \beta^{n(k_j)(s+\varepsilon)} \alpha^{-n(k_j)(d-\varepsilon)}  \\ \\
&\gtrsim& \left( \frac{\alpha^{k_j}}{\alpha^{n(k_j)+k_j}} \right)^{-\log m/\log\alpha -(s+\varepsilon)\log \beta/\log \alpha  +(d-\varepsilon)}
\end{eqnarray*}
which yields the desired lower bound upon letting $\varepsilon \to 0$.

\subsubsection{Picking up the Assouad dimension of $\pi F$}

Throughout this section let $d = \ad \pi F$.  In order to prove the second lower bound, it suffices to show that for all $\varepsilon>0$ there exist arbitrarily small $0<r'<R'$ with $(R'/r') \to \infty$ such that for some square $Q$ intersecting $F$ with side length comparable to $R'$ we have
\[
N_{r'}(Q \cap F) \geq (R'/r')^{d -  H/\log \alpha - \varepsilon}.
\]
Let $\varepsilon>0$ and choose $\textbf{i} \in \mathcal I^*$ such that
\[
\frac{  \log  \lvert [ \textbf{i} ] \rvert }{\lvert \textbf{i}  \rvert } \geq H-\varepsilon.
\]
Also, let $0<r<R<1$ and $I \subset [0,1]$ be an interval of length $R$ such that
\[
N_r(I \cap F) \gtrsim (R/r)^{d-\varepsilon}
\]
which can be done for arbitrarily small $0<r<R$ with $(R/r) \to \infty$ by the definition of $d$. For convenience, we also assume that $r/R \leq \alpha^{\lvert \textbf{i} \rvert}$.   Let $n \in \mathbb{N}$ be chosen (uniquely) to satisfy
\[
\alpha^{\lvert \textbf{i} \rvert(n+1)} < r/R \leq \alpha^{n\lvert \textbf{i} \rvert}.
\]
Consider $\textbf{i}^n$ (meaning the word $\textbf{i}$ concatenated with itself $n$ times) and observe that
\[
 \lvert [ \textbf{i}^n ] \rvert \geq   \lvert [ \textbf{i} ] \rvert^n.
\]
The collection $\{S_\textbf{j}(X) : \textbf{j} \in [ \textbf{i}^n ]\}$ is a collection of at least
\[
\exp(n \lvert  \textbf{i}  \rvert  (H-\varepsilon))
\]
rectangles of base length $\beta^{n \lvert \textbf{i} \rvert}$ and height $\alpha^{n \lvert \textbf{i} \rvert}$ which are all vertically aligned (i.e. all project to the same interval under $\pi$).  Consider the vertical strip of base $\beta^{n \lvert \textbf{i} \rvert}$ and height 1 which contains all of these rectangles and consider the thinner  sub-strip which lies above the appropriate image of $I$ (under the interval contraction induced by $S_{\textbf{i}^n}$) and label this thinner strip $T$.   This is a strip of base $\beta^{n \lvert \textbf{i} \rvert}R$ and height 1.  Choose $k \in \mathbb{N}$  (uniquely) to satisfy
\[
(\alpha/\beta)^{k+1} < \beta^{n \lvert \textbf{i} \rvert}R \leq (\alpha/\beta)^k
\]
and choose $\textbf{j} \in \mathcal I^k$ arbitrarily.  Let
\[
Q = S_{\textbf{j}}(T)
\]
and observe that $Q$ is a rectangle with  base length
\[
\beta^k \beta^{n \lvert \textbf{i} \rvert}R
\]
and height
\[
\alpha^k \asymp \beta^k \beta^{n \lvert \textbf{i} \rvert}R.
\]
Moreover, it is made up of at least $\exp(n\lvert \textbf{i} \rvert(H-\varepsilon))$  horizontal strips of height
\[
\alpha^{n \lvert \textbf{i} \rvert+k} \asymp \alpha^{n \lvert \textbf{i} \rvert}  \beta^k \beta^{n \lvert \textbf{i} \rvert}R.
\]
Write $R' = \beta^k \beta^{n \lvert \textbf{i} \rvert}R$ and $r' = \alpha^{n \lvert \textbf{i} \rvert}  \beta^k \beta^{n \lvert \textbf{i} \rvert}R$ and consider the problem of covering $Q\cap F$ by small squares of sidelength $r'$.   Since the height of each horizontal strip is comparable to $r'$, we may assume that each strip contributes individually.  Since each strip was chosen to lie above the appropriate image of $I$, we know that it will contribute $\gtrsim N_r(I \cap F)$ to any optimal cover (after rescaling the problem back to $I$ by a factor of $\beta^{-k} \beta^{-n \lvert \textbf{i} \rvert}$ and using the fact that $r \asymp R \alpha^{n \lvert \textbf{i} \rvert}$).  Therefore
\begin{eqnarray*}
N_{r'} (Q \cap F)  &\gtrsim&  \exp(n\lvert \textbf{i} \rvert(H-\varepsilon)) (R/r)^{d-\varepsilon} \\ \\
&\gtrsim&  (\alpha^{-n\lvert \textbf{i} \rvert})^{d-\varepsilon-(H-\varepsilon)/\log \alpha} \\ \\
&=&  (R'/r')^{d-\varepsilon-(H-\varepsilon)/\log \alpha}
\end{eqnarray*}
which yields the desired lower bound by observing that this estimate can be achieved for arbitrarily small $0<r'<R'$ with $R'/r' \to \infty$ by letting $R/r$ and thus $n$ tend to $\infty$ and letting $\varepsilon \to 0$.

\subsection{Upper bound}

The main difference between the upper and lower bound is that we must now handle all pairs $0<r<R$ and not just the sequence of pairs which is most convenient for the argument.  This makes the upper bound a little more subtle.

Throughout this section we  let $d = \ad \pi F$.  Let $0<r<R<1$ be arbitrary and choose $n,k \in \mathbb{N}$ (uniquely) to satisfy
\[
\alpha^{n+1} < r \leq \alpha^n.
\]
and
\[
\alpha^{k+1} < R \leq \alpha^k
\]
observing that $n \geq k$.  Consider an arbitrary $R$-square $Q$, and observe that it may only intersect at most a constant number of level $k$ cylinders (the height of which is approximately $R$) and so we may assume it intersects only one, $S_\textbf{j}(F)$ say.  Let
\[
M(n,k)  =  \# \{ \textbf{i} \in \mathcal I^n : S_{\textbf{i}}(F) \cap Q \neq \emptyset \} \asymp  \# \{ \textbf{i} \in \mathcal I^{n-k} : S_{\textbf{i}}(F) \cap S_\textbf{j}^{-1} Q \neq \emptyset \} .
\] 
Mirroring the lower bound we want to expand the vertical tube $S_\textbf{j}^{-1} Q$ so that any level $(n-k)$ cylinder contributing to $M(n,k)$ is  completely contained in the expanded tube.  However, there is a potential problem because the base length of a level $(n-k)$ cylinder is $\beta^{n-k}$ which is not necessarily comparable to the width of the original tube, which is $R \beta^{-k}$ (and  within a constant multiple of $(\alpha/\beta)^k$).  With this in mind, we do the best we can and expand the tube horizontally to have width
\[
(\alpha/\beta)^k + 2 \beta^{n-k}.
\]
We denote the expanded tube by $S_\textbf{j}^{-1} Q^*$ and let $B = \pi S_\textbf{j}^{-1} Q^*$ which is an interval inside $[0,1]$ intersecting $\pi F$ with diameter comparable to the  width of the expanded tube.  It follows that
\[
\pi \mu (B) \lesssim \left((\alpha/\beta)^k + 2 \beta^{n-k}\right)^{(s-\varepsilon)}
\]
with the implied constant independent of $n$ and $k$. Since each level $(n-k)$ cylinder carries weight $m^{-(n-k)}$ we also have the estimate
\[
\pi \mu (B) \geq M(n,k) m^{-(n-k)}
\]
which yields
\[
M(n,k) \lesssim m^{(n-k)} \left((\alpha/\beta)^k + 2 \beta^{n-k}\right)^{(s-\varepsilon)}.
\]
We wish to estimate $N_r(Q \cap F)$ and observing that each level $n$ cylinder intersecting $Q$ has height comparable to $r$ it again suffices to consider the contributions from each cylinder individually and then add them up.  Here we split our analysis into two cases. 
\\ \\  
\noindent  \textbf{Case 1: $\beta^{n-k}  \geq (\alpha/\beta)^k$}.
\\ \\
For $\textbf{i} \in \mathcal I^n$ we have for some $x \in \pi F$ that
\[
N_{r} \big( S_\textbf{i}(F) \cap Q \big)  \  \lesssim \     N_{r \beta^{-n}} \big( \pi F \cap B(x, R\beta^{-n}) \big) \  \lesssim \    \left(\frac{R\beta^{-n} }{r\beta^{-n} } \right)^{(d+\varepsilon)} \  \lesssim \    (\alpha^k/\alpha^n)^{(d+\varepsilon)}.
\]
Therefore
\begin{eqnarray*}
N_r(Q \cap F ) &\lesssim&  M(n,k)  (\alpha^{k-n})^{(d+\varepsilon)}\\ \\
  &\lesssim&  m^{(n-k)}  \beta^{(n-k)(s-\varepsilon)}  \alpha^{-(n-k)(d+\varepsilon)}
\end{eqnarray*}

\noindent  \textbf{Case 2: $\beta^{n-k}  < (\alpha/\beta)^k$}.
\\ \\
For $\textbf{i} \in \mathcal I^n$ we have
\[
N_{r} \big( S_\textbf{i}(F) \big)  \  \lesssim \    N_{r \beta^{-n}} \big( \pi F \big) \  \lesssim \    (r\beta^{-n})^{-(\bd \pi F+\varepsilon)} \  \lesssim \    (r\beta^{-n})^{-(d+\varepsilon)} \  \lesssim \    (\beta/\alpha)^{n(d+\varepsilon)}  .
\]
Therefore
\begin{eqnarray*}
N_r(Q\cap F) &\lesssim&  M(n,k)  (\beta/\alpha)^{n(d+\varepsilon)}\\ \\
 &\lesssim&  m^{(n-k)}  (\alpha/\beta)^{k(s-\varepsilon)}  (\beta/\alpha)^{n(d+\varepsilon)}\\ \\
  &=& m^{(n-k)}    (\beta/\alpha)^{n(d+\varepsilon)-k(s-\varepsilon)}\\ \\
  &\lesssim&  m^{(n-k)}  \beta^{(n-k)(s-\varepsilon)}  \alpha^{-(n-k)(d+\varepsilon)}
\end{eqnarray*}
where the final line uses the \textbf{Case 2} assumption and the fact that $d \geq s$.

In both cases we obtain
\begin{eqnarray*}
N_r(Q\cap F) &\lesssim&  m^{(n-k)}  \beta^{(n-k)(s-\varepsilon)}  \alpha^{-(n-k)(d+\varepsilon)}\\ \\
  &=&  \left( \alpha^{k}/\alpha^n \right)^{-\log m/\log\alpha -(s-\varepsilon)\log \beta/\log \alpha  +(d+\varepsilon)} \\ \\
    &\lesssim&  \left( \frac{R}{r}\right)^{-\log m/\log\alpha -(s-\varepsilon)\log \beta/\log \alpha  +(d+\varepsilon)}
\end{eqnarray*}
which yields the desired upper bound upon letting $\varepsilon \to 0$.

\section{Extensions and further work} \label{further}

It would clearly be interesting to push this work further and, in particular, to consider more general families of self-affine carpets, such as those considered by Feng and Wang \cite{fengaffine}.  Our main restriction was in assuming that the linear parts of all the defining affine maps were the same, i.e., of the form $(x,y) \mapsto (\beta x, \alpha y)$ for some fixed $0<\alpha<\beta<1$.  Our results depend crucially on this assumption.  The reason for this is that we need the measure of a cylinder to tell us both the height \emph{and} the width of the corresponding construction rectangle.  More precisely, the Assouad dimension was linked to the $\pi \mu$-measure of an interval in the projection by choosing the measure in such a way that the measure of an interval told us exactly what we see above it.  If the linear parts are different (even if the contractions in one of the directions are fixed) one may find very differently shaped construction rectangles projecting to intervals of the same measure, or even of the same measure and the same length, and so things would be much more complicated in this setting. In general it looks like finding good upper bounds for the Assouad dimension of self-affine sets is a very challenging problem.

\vspace{8mm}

\begin{centering}
\textbf{Acknowledgements}\\
This work began while both authors were participating in the ICERM Semester Program on \emph{Dimension and Dynamics} and are grateful for the stimulating atmosphere they found there.  The authors thank P\'eter Varj\'u,  De-Jun Feng and Xiong Jin for helpful discussions.  J.M.F. is financially supported by a \emph{Leverhulme Trust Research Fellowship}.
\end{centering}

\begin{multicols}{2}{

\noindent Jonathan M. Fraser\\
School of Mathematics and Statistics \\
The University of St Andrews\\
St Andrews\\
KY16 9SS\\
Scotland\\
\emph{Email: jmf32@st-andrews.ac.uk} \\

\noindent Thomas Jordan\\
School of Mathematics\\
University of Bristol\\
Bristol\\
BS8 1TW\\
\emph{Email: thomas.jordan@bristol.ac.uk} \\

}

\end{multicols}

\end{document}